\ifx \processToSlides \undefined
  \ifx\usepackage\RequirePackage        % documentclass already done...  Assume
    \let\usingAmsArtXII\usepackage      % this file is processed by amsart
  \fi
\else
  \documentclass[12pt]{amsart}
  \usepackage{amssymb,amscd}

\usepackage[all]{xy}
  \usepackage[letterpaper,margin=2cm,includeheadfoot]{geometry}
  \def \useHugeSize {}
  \def \numberingIsThrough {}
\fi

\ifx \labelsONmargin\undefined

  \ifx\RequirePackage\undefined
    \expandafter\ifx\csname amsart.sty\endcsname\relax
        \documentstyle[12pt,amssymb,amscd]{amsart}
    \fi

    \def\mathbb{\Bbb}
    \def\mathfrak{\frak}
    \def\mathbf{\bold}
    \ifx\boldsymbol\undefined   % I *think* some very old AMS* did not have it
      \def\boldsymbol#1{{\bold #1}}
    \fi

    \def\mathbit{\boldsymbol}

    \makeatletter
    \newenvironment{proof}{%
         \@ifnextchar[{%
                       \expandafter\let\expandafter\end@proof
                         \csname endpf*\endcsname
                         \my@proof
                      }{\let\end@proof\endpf\pf}%
        }{\end@proof}
    \def\my@proof[#1]{\@nameuse{pf*}{#1}}
    \makeatother

    \def\xrightarrow[#1]#2{@>{#2}>{#1}>}
    \def\xleftarrow[#1]#2{@<{#2}<{#1}<}

    \def\providecommand#1{\def#1}
    \def\emph#1{{\em #1}}
    \def\textbf#1{{\bf #1}}

    \def\mathring{\overset{\,\,{}_\circ}}% For slanted letters only, sub too high
  \else
      \ifx\usepackage\RequirePackage
      \else     %% <<<<<<<<<<<<<<<<<<<<<<<<<<<<<<<< THIS IS THE MAIN BRANCH
        \documentclass[12pt]{amsart}            %% {{{{{{{{ CHOSE ONE OF TWO:
        \usepackage{amssymb,amscd}
        \let\usingAmsArtXII\usepackage

      \fi
      \ifx \mathring\undefined          % Not AmS-LaTeX 2
        \DeclareMathAccent{\mathring}{\mathalpha}{operators}{"17}
      \fi

      \long\def\FAKEendPROOF{\endtrivlist}
      \ifx\endproof\FAKEendPROOF
          \def\endproof{\qed\endtrivlist}
      \fi

      \ifx\mathbit\undefined    % They can collect their senses and define it
        \DeclareMathAlphabet{\mathbit}{OML}{cmm}{b}{it}
      \fi

      \ifx \usingAmsArtXII \undefined
      \else                             % We loaded this package ourselves
        \advance\oddsidemargin -0.1\textwidth
        \evensidemargin\oddsidemargin
      \fi

      \def\Sb#1\endSb{_{\substack{#1}}}
      \def\Sp#1\endSp{^{\substack{#1}}}

  \fi
\makeatletter
\@ifundefined{DeclareFontShape} 
     {\@ifundefined{selectfont}
             {%    Here we are in non-NFSS
                \message{We need AmS-LaTeX, this version of LaTeX does not 
                        support it}\@@end
             }{%  Here we are in  NFSS1
                \def\mathcal{\cal}
                
                \def\pcyr{%
                        \def\default@family{UWCyr}%
                        \let\oldSl@\sl
                        \def\sl{\def\default@shape{it}\oldSl@}%
                        \cyracc
                        \language\Russian\family{UWCyr}\selectfont
                }
             }}{% Here we are in  NFSS2
                \DeclareFontEncoding{OT2}{}{\noaccents@}
                \DeclareFontSubstitution{OT2}{cmr}{m}{n}
                \DeclareFontFamily{OT2}{cmr}{\hyphenchar\font45 }
                \DeclareFontShape{OT2}{cmr}{m}{n}{%
                     <5><6><7><8><9><10>gen*wncyr %
                     <10.95><12><14.4><17.28><20.74><24.88> wncyr10 %
                }{}
                \DeclareFontShape{OT2}{cmr}{m}{it}{%
                     <5><6><7><8><9><10> gen * wncyi%
                     <10.95><12><14.4><17.28><20.74><24.88> wncyi10%
                }{}
                \DeclareFontShape{OT2}{cmr}{bx}{n}{%
                     <5><6><7><8><9><10> gen * wncyb%
                     <10.95><12><14.4><17.28><20.74><24.88> wncyb10%
                }{}
                \DeclareFontShape{OT2}{cmr}{m}{sl}{%
                     <-> ssub * cmr/m/it%
                }{}
                \DeclareFontShape{OT2}{cmr}{m}{sc}{%
                     <5><6><7><8><9><10>%
                     <10.95><12><14.4><17.28><20.74><24.88> wncysc10%
                }{}
                \DeclareFontFamily{OT2}{cmss}{\hyphenchar\font45 }
                \DeclareFontShape{OT2}{cmss}{m}{n}{%
                     <8><9><10> gen * wncyss%
                     <10.95><12><14.4><17.28><20.74><24.88> wncyss10%
                }{}
                
                \def\cyrencodingdefault{OT2}
                \def\pcyr{%
                        \cyracc
                        \let\encodingdefault\cyrencodingdefault
                        \language\Russian\fontencoding{OT2}\selectfont
                }
             }   

\@ifundefined{theorembodyfont} 
     {
        \def\theorembodyfont#1{\relax}
        \makeatletter
          \let\@@th@plain\th@plain
          \def\th@plain{ \@@th@plain \slshape }
        \makeatother
        \let\normalshape\relax
     }{}   

\ifx\cprime\undefined
     \def\cprime{$'$}
\fi
\makeatletter
  \def\@sect@my#1#2#3#4#5#6[#7]#8{%
\ifnum #2>\c@secnumdepth
   \let\@svsec\@empty
 \else
   \refstepcounter{#1}%
\edef\@svsec{\ifnum#2<\@m
             \@ifundefined{#1name}{}{\csname #1name\endcsname\ }\fi
\noexpand\rom{\csname the#1\endcsname.}\enspace}\fi
 \@tempskipa #5\relax
 \ifdim \@tempskipa>\z@ % then this is not a run-in section heading
   \begingroup #6\relax
   \@hangfrom{\hskip #3\relax\@svsec}{\interlinepenalty\@M #8\par}%
   \endgroup
   \if@article\else\csname #1mark\endcsname{%
        \ifnum \c@secnumdepth >#2\relax\csname the#1\endcsname. \fi#7}\fi
\ifnum#2>\@m \else
       \let\@tempf\\ \def\\{\protect\\}\addcontentsline{toc}{#1}%
{\ifnum #2>\c@secnumdepth \else
             \protect\numberline{%
               \ifnum#2<\@m
               \@ifundefined{#1name}{}{\csname #1name\endcsname\ }\fi
               \csname the#1\endcsname.}\fi
           #8}\let\\\@tempf
     \fi
 \else
  \def\@svsechd{#6\hskip #3\@svsec
    \@ifnotempty{#8}{\ignorespaces#8\unskip
       \ifnum\spacefactor<1001.\fi}%
        \ifnum#2>\@m \else
          \let\@tempf\\ \def\\{\protect\\}\addcontentsline{toc}{#1}%
            {\ifnum #2>\c@secnumdepth \else
              \protect\numberline{%
                \ifnum#2<\@m
                \@ifundefined{#1name}{}{\csname #1name\endcsname\ }\fi
                \csname the#1\endcsname.}\fi
             #8}\let\\\@tempf\fi}%
 \fi
\@xsect{#5}}
  \ifx\RequirePackage\undefined
  \let\@sect\@sect@my             % Cannot just comment the above
  \fi
  \ifx\RequirePackage\undefined %              Not under LaTeX2e
  \def\th@remark@my{\theorempreskipamount6\p@\@plus6\p@
    \theorempostskipamount\theorempreskipamount
    \def\theorem@headerfont{\it}\normalshape}
    \let\th@remark\th@remark@my
  \else                         %              Under LaTeX2e
    \let\o@@remark\th@remark

    \ifx\theorempostskipamount\undefined                % AmS-art
    \else                                               % Transformation groups
      \def\th@remark{\o@@remark
        \ifdim\theorempostskipamount < 2pt\relax
          \theorempostskipamount\theorempreskipamount
             \multiply\theorempostskipamount\tw@
             \divide\theorempostskipamount\thr@@
        \fi
      }
    \fi
  \fi
\let\myLabel\@gobble
\def\labelsONmargin{\@mparswitchfalse\def\myLabel##1{\@bsphack\marginpar
                                  {\normalshape\tiny\rm Label ##1}\@esphack}}
\ifx \url\undefined
  \def\url#1{{\tt #1}}%
\fi
\def\PREpmodSKIP{\allowbreak  \if@display\mkern18mu\else\mkern8mu\fi}

\def\cyracc{\def\u##1{%\if \i##1\accent"24 i%
                \if \i##1\char"1A%
                \else \if I##1\char"12%
                \else \accent"24 ##1\fi\fi }%
\def\"##1{\if e##1{\char"1B}%
                \else \if E##1{\char"13}%
                \else \accent"7F ##1\fi\fi }%
\def\9##1{\if##1z\char"19 
\else\if##1Z\char"11 
\else\if##1E\char"03 
\else\if##1e\char"0B 
\else\if##1u\char"18 
\else\if##1U\char"10 
\else\if##1A\char"17 
\else\if##1a\char"1F 
\else\if##1p\char"7E 
\else\if##1P\char"5E 
\else\if##1Q\char"5F 
\else\if##1q\char"7F 
\else\if##1i\char"1A 
\else\if##1I\char"12 
\else\if##1N\char"7D 
\fi
\fi
\fi
\fi
\fi
\fi
\fi
\fi
\fi
\fi
\fi
\fi
\fi
\fi
\fi
}%
\def\cydot{{\kern0pt}}}%
\def\cydot{$\cdot$}
\ifx\Russian\undefined
        \def\Russian{0\relax
    \message{Don't know the hyphenation rules for Russian^^J
                        Please do INITeX with `input  russhyph' in the 
                        command line}%
                \gdef\Russian{0\relax}%
        }
\fi
\ifx\RequirePackage\undefined                   % for 209
  \def\@putname#1#2#3#4{\def\@@ref{#3}\let\old@bf\bf
        \def\bf##1{\old@bf\if?\noexpand##1?{#4}\else##1\fi}%
        #1{#2}%
        \let\bf\old@bf}
\else                                           % for 2e
  \def\@putname#1#2#3#4{\def\@@ref{#3}\let\old@bf\bf    % for 209???
        \let\old@reset@font\reset@font                  % for 2e
        \def\bf##1{\old@bf\if?\noexpand##1?{#4}\else##1\fi}%
        \def\reset@font##1##2{\old@reset@font##1\if?\noexpand##2?{#4}\else##2\fi}#1{#2}%
        \let\bf\old@bf\let\reset@font\old@reset@font}
\fi
\let\my@ref=\ref
\def\ref#1{\@putname\my@ref{#1}{#1}{\tiny\rm\@@ref}}
\let\my@pageref=\pageref
\def\pageref#1{\@putname\my@pageref{#1}{#1}{\tiny\rm\@@ref}}
\let\my@cite=\cite
\def\cite#1{\@putname\my@cite{#1}{\@citeb}{\tiny\rm\@@ref}}
\makeatother
\fi
\relax 

\ifx \SKIPstatementDEFS \undefined
  \theoremstyle{plain} % for references in unnumbered theorems
  \theorembodyfont{\sl}
\fi

\ifx\useDoubleSpacing\undefined
\else
  \usepackage{setspace}
\fi

\ifx \address \undefined
  \if \institute \undefined \else       % Springer Verlag multiauthor books
     \def\address{\institute}
     \ifx \email \undefined
        \let\email\texttt
     \fi
  \fi
\fi

\let\emphOrig\emph
\ifx \doEmphToIndex \undefined
  
\else
  \usepackage{makeidx}
  \makeindex

  \def\eatToBar#1|{}
  \def\emphToIndexSLASH#1\/{\index{#1}\eatToBar}
  \def\emphToIndexDOTSLASH#1.\/{\emphToIndexSLASH #1\/}
  \def\emphAndIndex#1{\emphOrig{#1}{\emphToIndexDOTSLASH #1.\/|}}
  \let\emph\emphAndIndex
\fi

\ifx \SKIPstatementDEFS \undefined

\ifx \numberingIsThrough \undefined
\numberwithin{equation}{section}
\fi

\theorembodyfont{\rm}
\theoremstyle{definition}

\ifx \numberingIsThrough \undefined
\newtheorem{definition}{Definition}[section]
\else

\fi

\theoremstyle{remark}

\ifx \numberingIsThrough \undefined
\newtheorem{note}{Note}[section] 
\newtheorem{summary}{Summary}[section] 
\else

\fi

\theoremstyle{plain} % for future references
\theorembodyfont{\sl}

\newtheorem*{theor*}{Theorem}

\newcommand{\fg}{\mathfrak{g}}

\newcommand{\fk}{\mathfrak{k}}

\newcommand{\fn}{\mathfrak{n}}
\newcommand{\fb}{\mathfrak{b}}
\newcommand{\fm}{\mathfrak{m}}

\newcommand{\fh}{\mathfrak{h}}
\fi %%% end: if \SKIPstatementDEFS \undefined

\usepackage[mathscr]{euscript}
\usepackage{enumerate,color}

\begin{document}
\bibliographystyle{amsplain}

\ifx\useHugeSize\undefined
\else
\Huge
\fi

\relax

\author[Ivan Penkov]{\;Ivan Penkov}

\address{
Ivan Penkov
\newline Jacobs University Bremen
\newline Campus Ring 1
\newline 28759 Bremen, Germany}
\email{i.penkov@jacobs-university.de}

\author[Gregg Zuckerman]{\;Gregg Zuckerman}

\address{
Gregg Zuckerman
\newline Department of Mathematics
\newline Yale University
\newline 10 Hillhouse Avenue, P.O. Box 208283
\newline New Haven, CT 06520-8283, USA}
\email{gregg.zuckerman@yale.edu}

\title{On the existence of infinite-dimensional generalized Harish-Chandra modules}

\date{ \today }

\maketitle

\begin{center}
\it To our friend Joe
\end{center}

\begin{abstract} We prove a general existence result for infinite-dimensional admissible $(\fg,\fk)$-modules, where $\fg$ is a reductive finite-dimensional complex Lie algebra and $\fk$ is a reductive in $\fg$ algebraic subalgebra.
\end{abstract}


\medskip

In this note, we draw a corollary of our earlier work \cite{PZ1}. In the subsequent works \cite{PZ2}, \cite{PZ3}, \cite{PZ4}, \cite{PSZ} we have built foundations of an algebraic theory of generalized Harish-Chandra modules.

The base field is $\mathbb{C}$. Let $\fg$ be a finite-dimensional (complex) reductive Lie algebra and let $\fk \in \fg$ be a reductive in $\fg$ algebraic subalgebra. A ($\fg, \fk$)-\emph{module} $M$ is a $\fg$-module $M$ on which $\fk$ acts locally finitely, i.e. dim ($U(\fk)\cdot m) < \infty$ for any $m\in M$. Under the assumption that $M$ is a simple $\fg$-module, the requirement that $M$ be a ($\fg,\fk$)-module is equivalent to the requirement that as a $\fk$-module $M$ decomposes into a direct sum of simple finite-dimensional $\fk$-modules. An \emph{admissible}  ($\fg,\fk$)-module $M$ is a  ($\fg,\fk$)-module which, after restriction to $\fk$, is isomorphic to a direct sum of simple finite-dimensional $\fk$-modules with finite multiplicities.

Both of these notions go back to the 1960's. By a \emph{generalized Harish-Chandra module}, we understand a $\fg$-module $M$ for which there exists a reductive subalgebra $\fk$ of $\fg$ such that $M$ is an admissible  ($\fg,\fk$)-module. The case of \emph{Harish-Chandra modules} corresponds to the case where  $\fk$ is a symmetric subalgebra of $\fg$. Under this latter assumption, there is an extensive literature on $(\fg,\fk)$-modules, and here we just direct the reader to \cite{V} and \cite{KV} and references therein.

The question of interest in the present note is the following:

What is a necessary and sufficient condition on an algebraic, reductive in $\fg$ subalgebra $\fk$ for the existence of a simple infinite-dimensional admissible  ($\fg,\fk$)-module?

We know of no published answer to this question. However, we have observed that the answer is actually implicit in our work. To make it explicit, we  prove the following.

\begin{theor*}
For an algebraic reductive in $\fg$ subalgebra $\fk$, there exists a simple infinite-dimensional admissible  ($\fg,\fk$)-module if and only if $\fk$ is not an ideal of $\fg$. 
\end{theor*}

\begin{proof}
If $\fk$ is an ideal of $\fg$, then any simple  ($\fg,\fk$)-module $M$ is isomorphic to an outer tensor product $M_{\fk}\boxtimes M'$, where $M_{\fk}$ is a simple finite-dimensional $\fk$-module and $M'$ is a simple module over a direct complement $\fg'$ of $\fk$ in $\fg$. Indeed, fix a simple finite-dimensional $\fk$-submodule $M_{\fk}$ of $M$ (which exists because of the locally finite action of $\fk$ on $M$). Then the isotypic component of $M_{\fk}$ in $M$ is  a $\fg$-submodule since $\fk$ is an ideal in $\fg$. Therefore,
$$M=M_{\fk}\otimes\textup{Hom}_{\fk}(M_{\fk},M).$$
Setting $M':=\textup{Hom}_{\fk}(M_{\fk},M)$, we see that the simplicity of $M$ as a $\fg$-module implies the simplicity of $M'$ as a $\fg'$-module.
 Consequently, if $M$ is admissible then $M$ is finite dimensional.

Assume now that $\fk$ is not an ideal in $\fg$. Without loss of generality we can assume that $\fg$ is semisimple and that $\fk$ does not contain an ideal of $\fg$. We have $\fg=\fk\oplus\fk^{\perp}$ where $^\perp$ indicates orthogonal space with respect to the Killing form $\langle\cdot,\cdot\rangle$ on $\fg$. We fix a Cartan subalgebra $\mathfrak{t}$ of $\fk$ and an element $h\in \mathfrak{t}$ which is regular in $\fk$ and has real eigenvalues in  $\fg$. By $\fg^{\alpha}$ we denote the eigenspaces of $h$ in $\fg$.  Then 
$$ \mathfrak{p}:=C_{\fg}(h)\oplus(\bigoplus_{\alpha>0}\fg^{\alpha})$$
is a \emph{minimal} $\mathfrak{t}$-\emph{compatible parabolic subalgebra} of $\fg$. Here $C_{\fg}(h)$ is the centralizer of $h$ in $\fg$. The notions of $\mathfrak{t}$-compatible and minimal $\mathfrak{t}$-compatible parabolic subalgebra are discussed in \cite{PZ1}. In what follows, we set $\mathfrak{m}:=C_{\fg}(h)$ and $\fn:=\bigoplus_{\alpha>0}\fg^\alpha$ and note that in the semidirect sum $\mathfrak{p}=\mathfrak{m}\oplus\fn$, $\mathfrak{m}$ is the reductive part of $\mathfrak{p}$ and $\fn$ is the nilradical of $\mathfrak{p}$. Furthermore, $\fk^{\perp}=(\fn\cap\fk^{\perp})\oplus(\mathfrak{m}\cap\fk^{\perp})\oplus(\bar{\fn}\cap\fk^{\perp})$ where $\bar{\fn}:=\bigoplus_{\alpha<0}\fg^{\alpha}$. The assumption that $\fk$ does not contain an ideal of $\fg$ implies that $h$ does not commute with $\fk^{\perp}$, and hence
$$\fn\cap\fk^{\perp}\neq0.$$
In particular,
$$r:=\textup{dim}(\fn\cap\fk^{\perp})>0.$$

Fix a Borel subalgebra $\fb$ of $\fg$ such that $\fb\subset\mathfrak{p}$ and $\fb\cap\fk$ is a Borel subalgebra of $\fk$. Recall from \cite{PZ1} the construction of the fundamental series ($\fg,\fk$)-module $F^s(\mathfrak{p}, E)$ where $E$ is a finite-dimensional simple $\mathfrak{p}$-module which, as an $\mathfrak{m}$-module, has highest weight $\nu$ with respect to the Borel subalgebra $\fb_{\mathfrak{m}}=\fb\cap\mathfrak{m}$ of $\mathfrak{m}$. Note that rk$\mathfrak{m}$$=$rk$\mathfrak{g}$, and assume that $\nu\in\fh^*$ for a fixed Cartan subalgebra $\fh$ of $\fg$ lying in $\mathfrak{m}$ and containing $\mathfrak{t}$. Let $\omega\in \mathfrak{t}^*$ be the restriction of $\nu$ to $\mathfrak{t}$. By $\mu$ we denote the $\mathfrak{t}$-weight $\omega+2\rho^{\perp}_{\fn}$, where $\rho^{\perp}_{\fn}$ is the half-sum of the $\mathfrak{t}$-weights of $\fn\cap\mathfrak{k}^{\perp}$ with multiplicities, that is, the half-sum of the multiset of $\mathfrak{t}$-weights of $\fn\cap\mathfrak{k}^{\perp}$.

In what follows, we assume that $\mu$ is an integral weight of $\mathfrak{k}$, dominant with respect to the Borel subalgebra $\fb\cap\mathfrak{k}$ of $\fk$. We need one further assumption on $\mu$.

Following \cite{PZ1}, we call $\mu$ \emph{generic} if the following two conditions are satisfied:

\begin{enumerate}
\item $\langle \textup{Re}\mu +2\rho-\rho_{\fn}, \beta\rangle\geq 0$ for every $\mathfrak{t}$-weight $\beta$ of $\fn\cap\mathfrak{k}$,
\item $\langle \textup{Re}\mu +2\rho-\rho_{S}, \rho_{S}\rangle>0$ for every submultiset of the multiset $S$ of $\mathfrak{t}$-weights of $\fn$,
\end{enumerate}
where $\rho$ is the half-sum of the $\mathfrak{t}$-roots of $\mathfrak{k}$ and $\rho_{\fn}$ is the half-sum of the multiset of $\mathfrak{t}$-weights of $\fn$.

Theorem 2 of \cite{PZ1} implies that, under the additional assumption of genericity of $\mu$ (which is ultimately a condition on $\nu$), the ($\fg,\fk$)-module $F^s(\mathfrak{p},E)$ is a nonzero admissible ($\fg,\fk$)-module with a unique simple submodule $\bar{F}^s(\mathfrak{p},E)$. Here $s=\textup{dim}(\fk\cap\fn)$. Moreover,  Proposition 6 in \cite{PZ1} claims  that there is an isomorphism of vector spaces

$$\textup{Hom}_\fg(M,\bar{F}^s(\mathfrak{p},E))\simeq\textup{Hom}_{\mathfrak{m}}(H^r(\fn,M),E),$$
for any simple admissible ($\fg,\fk$)-module $M$. We know that  	 $\bar{F}^s(\mathfrak{p},E)$ is a simple admissible  ($\fg,\fk$)-module, so it only remains to show the existence of a weight $\nu$ which satisfies all above assumptions and such that dim$\bar{F}^s(\mathfrak{p},E)=\infty$.

We consider two possibilities. Either there exists a $\nu$ as above such that the central character of the ($\fg,\fk$)-module $\bar{F}^s(\mathfrak{p},E)$ is not integral, or the central character of $\bar{F}^s(\mathfrak{p},E)$  is necessarily integral (as a consequence of all our assumptions on $\nu$). In the former case, we are done as then necessarily dim$\bar{F}^s(\mathfrak{p},E)=\infty$. In the latter case we will further assume that $\nu$ is integral $\fb$-dominant for $\fg$. Lemma 2.3 in \cite{PZ3} shows that this additional assumption is compatible with all previous  assumptions on $\nu$. Then, by Theorem $2\,\textup{c})$ in \cite{PZ1},  the simple finite-dimensional $W$ with $\fb$-highest weight $\nu$ is the only (up to isomorphism) simple finite-dimensional module whose central character coincides with that of $\bar{F}^s(\mathfrak{p},E)$.

Therefore it suffices to show that
$$\textup{Hom}_\fg(W,\bar{F}^s(\mathfrak{p},E))=\textup{Hom}_{\mathfrak{m}}(H^r(\fn,W),E)=0.$$
  For this, recall that Kostant's Theorem \cite{K} asserts that there is an  isomorphism of $\fm$-modules
$$H^r(\fn,W)\simeq\bigoplus_w E(w(\nu+\tilde\rho)-\tilde\rho).$$
Here $\tilde\rho$ is the half-sum of roots of $\fb$, $E(\gamma)$ is a simple $\fm$-module with highest weight $\gamma$, and the sum is taken over all elements $w$ of the Weyl group of $\fg$ of length $r$ for which the weights $w(\nu+\tilde\rho)-\tilde\rho$ are $\fb_{\fm}$-dominant. Since $r>0$, we infer that 
$$\textup{Hom}_{\mathfrak{m}}(H^r(\fn,W),E)=0,$$
and the theorem is proved.
\end{proof}

In conclusion, we would like to make two brief comments on how the modules, whose existence is claimed in the above theorem, fit into the panorama of well-studied (and not so well-studied) $\fg$-modules. Our first remark is that  $\bar{F}^s(\mathfrak{p},E)$ does not have to be a ($\fg,\fk'$)-module for any reductive in $\fg$ subalgebra $\fk'$ which contains $\fk$ properly. Indeed, let $\fg=\mathfrak{sl}(n)$ for $n\geq4$ and let $\fk$ be a principal $\mathfrak{sl}(2)$-subalgebra. By the same argument as in our expository paper \cite{PZ0}, using  the work of Willenbring and the second author \cite{WZ}, one can show that the $\fg$-module $\bar{F}^1(\mathfrak{p},E)$ (here $s=1$) is not a ($\fg,\fk'$)-module for any $\fk'$ as above. In particular, $\bar{F}^1(\mathfrak{p},E)$ is not a Harish-Chandra module for the pair ($\fg,\mathfrak{so}(n)$) if $n=2k+1$, or the pair ($\fg,\mathfrak{sp}(n)$) if $n=2k$.

Our second comment is that $\fk$ does not have to be a symmetric subalgebra of $\fg$ for the module $\bar{F}^s(\mathfrak{p},E)$ to be  a Harish-Chandra module. For instance, let $\fg$ be simple and $\fk$ be an ideal in a symmetric subalgebra $\fk'$ of $\fg$. Then $\bar{F}^s(\mathfrak{p},E)$ is an admissible $(\fg,\fk)$-module which is also a ($\fg,\fk'$)-module,  hence a Harish-Chandra module. The property of  a Harish-Chandra module to be admissible over an ideal of the relevant symmetric subalgebra has been studied in the literature. This applies in particular to the work of Orsted and Wolf \cite{OW}, where certain ideals of symmetric subalgebras are singled out and discrete series modules -$ $-admissible over these ideals-$ $- are investigated. Our approach in \cite{PZ1} provides an alternative construction which applies to any ideal of a symmetric subalgebra but, even in the case of Orsted and Wolf, the range of Harish-Chandra modules arising through this construction  requires further study.

{\bf Acknowledgment.} We acknowledge the hospitality of the American Institute of Mathematics in San Jose where this paper was conceived during a SQuaRE meeting. IP has been supported in part by DFG grant PE 980/6-1.

\end{document}